\documentclass[12pt,reqno]{amsart}   	% use "amsart" instead of "article"
\usepackage[letterpaper, margin=1in]{geometry}
\usepackage{graphicx}				% Use pdf, png, jpg, or epsÂ§ with pdflatex; use eps in DVI mode
\usepackage{amssymb}
\usepackage{mathtools,amsmath}
\usepackage{amsthm}
\usepackage{amssymb}
\usepackage{mathrsfs}
\usepackage{epstopdf}
\usepackage{color}
\usepackage{enumerate}
\usepackage[pagebackref,colorlinks,citecolor=blue,linkcolor=magenta]{hyperref}
\usepackage[capitalise]{cleveref}
\usepackage{tikz}

\usepackage[linesnumbered,ruled]{algorithm2e}
\RequirePackage{amsthm,amsmath,amsfonts,amssymb}
\usepackage[utf8]{inputenc}
\usepackage{chngcntr}
\usepackage{float}

\theoremstyle{plain}
\newtheorem{prop}{Proposition}[section]
\newtheorem{thm}[prop]{Theorem}
\newtheorem{cor}[prop]{Corollary}
\newtheorem{lem}[prop]{Lemma}
\newtheorem{conj}[prop]{Conjecture}
\newtheorem{question}[prop]{Question}
\newtheorem*{thm*}{Theorem}

\theoremstyle{definition}
\newtheorem{dfn}[prop]{Definition}
\newtheorem{rem}[prop]{Remark}
\newtheorem{example}[prop]{Example}

\renewcommand{\iff}{\Leftrightarrow}

\newcommand{\C}{{\mathbb{C}}}
\renewcommand{\P}{{\mathbb{P}}}
\newcommand{\R}{{\mathbb{R}}}

\newcommand{\N}{{\mathbb{N}}}
\newcommand{\Z}{{\mathbb{Z}}}

\newcommand{\OO}{{\mathcal{O}}}

    \DeclareMathOperator{\interior}{int}

\DeclareMathOperator{\gl}{GL}

\DeclareMathOperator{\spn}{span}
\DeclareMathOperator{\codim}{codim}

\DeclareMathOperator{\dphi}{d\phi}

\DeclareMathOperator{\cone}{cone}

\newcommand{\cc}{change of coordinates}
\DeclareMathOperator{\face}{\mathfrak{F}}

\newcommand{\Rx}{\mathbb{R}[\ul{x}]}

\newcommand{\vc}{\mathcal{V}}

\newcommand{\du}{{\scriptscriptstyle\vee}}
\renewcommand{\setminus}{\smallsetminus}
\newcommand{\ol}{\overline}
\newcommand{\ul}{\underline}

\newcommand{\all}{\forall\,}

\renewcommand{\subset}{\subseteq}
\renewcommand{\supset}{\supseteq}
\newcommand{\bil}[2]{\langle{#1},{#2}\rangle}

\newcommand{\id}[1]{\langle #1 \rangle}
\newcommand{\sa}{semi-alge\-braic}
\newcommand{\py}{p}
\newcommand{\bin}{\Sigma_{2,8}^4}
\newcommand{\ter}{\Sigma_{3,4}^2}

\newcommand{\minus}{\text{-}}

\renewcommand{\setminus}{\smallsetminus}
\renewcommand{\epsilon}{\varepsilon}
\renewcommand{\theta}{\vartheta}

\newcommand{\todfn}[1]{\textit{#1}}

\title{On higher Pythagoras numbers of real polynomial rings}
\author{}
\date{\today}

\author{Tomasz Kowalczyk}
\address{Institute of Mathematics, Faculty of Mathematics and Computer Science, Jagiellonian University, ul. Łojasiewicza 6, 30-348 Kraków, Poland}
\email{tomek.kowalczyk@uj.edu.pl}

\author{Julian Vill}
\address{Fakult\"at f\"ur Mathematik, Otto-von-Guericke Universit\"at Magdeburg, Universitätsplatz 2, 39106 Magdeburg, Germany}
\email{julian.vill@ovgu.de}

\begin{document}


\begin{abstract}
We show that the higher Pythagoras numbers for the polynomial ring are infinite  $p_{2s}(K[x_1,x_2,\dots,x_n])=\infty$ provided that $K$ is a formally real field, $n\geq2$ and $s\geq 1$. This almost fully solves an old question \cite[Problem 8]{cldr1982}. The remaining open cases are precisely $n=1$ and $s>1$. Moreover, we study in detail the cone of binary octics that are sums of fourth powers of quadratic forms. We determine its facial structure as well as its algebraic boundary. This can also be seen as sums of fourth powers of linear forms on the second Veronese of $\P^1$. As a result, we disprove a conjecture of Reznick \cite[Conjecture 7.1]{reznick2011} which states that if a binary form $f$ is a sum of fourth powers, then it can be written as $f=f_1^2+f_2^2$ for some nonnegative forms $f_1,f_2$.
\end{abstract}

\maketitle

\section{Introduction}

This paper is devoted to the qualitative study of sums of higher powers, mainly in polynomial rings. In particular, we are interested in the problem of computing the higher Pythagoras numbers of a given ring.

It is known that the theory of quadratic forms is well developed. Literature devoted to that topic is extremely vast. Every quadratic form over a field of characteristic not two is diagonalizable. This is not true for forms of higher degree (cf. \cite{harrison1975}). In particular considering forms of degree at least $3$ one has to distinguish the cases of diagonal and general forms.

Let $A$ be a commutative ring with identity and $s>1$ be a positive integer.
\begin{dfn}
We define the $s$-th Pythagoras number of $A$,  $p_{s}(A)$, to be the least positive integer $\ell$ such that any sum of $s$-th powers of elements of $A$ can be expressed as a sum of $\ell$ $s$-th powers of elements of $A$. If such a number does not exist, we put $p_{s}(A)=\infty$.
\end{dfn}

Cassels Theorem \cite{cassels1964} states that if a quadratic form with coefficients from a field $K$ represents a polynomial $f \in K[x]$ in a single variable over the field $K(x)$ then it already represents $f$ over $K[x]$. Taking the above quadratic form to be a sum of squares (sos) we deduce the equality of $2$-Pythagoras numbers $p_2(K[x])=p_2(K(x))$. Cassels Theorem does not generalize to higher powers \cite{prestel1984}. In particular, a polynomial in a single variable can be a sum of fourth powers of rational functions, but not a sum of fourth powers of polynomials. Currently, this appears to be the main obstruction for computing $p_{2s}(\R[x])$ (cf. \cref{sec:Open problems}), however it is known that $p_{2s}(\R(x))$ is finite \cite{becker1982}.

Let us also mention one more difference. Consider the usual $n$-dimensional unit sphere $\mathbb{S}^n$. Obviously, it is the zero set of the quadratic polynomial $\sum_{i=1}^{n+1} x_i^2 -1=0$. If one further considers the group of linear isometries of $\mathbb{S}^n$, then it is known that this action is transitive.
Consider now an algebraic set given by the zero set of $\sum_{i=1}^{n+1}x_i^{2s}-1=0$. There are very few linear isometries of such a set if $s>1$.  In particular, the action of linear isometries is not transitive. This follows from the fact that $2s$-norm of $\ell^{2s}(\R^n)$ is not induced by the inner product for $n>1$ (see also \cite{ding2003, ls1994} for the structure of linear isometries of $\ell^{2s}(\R^n)$). Necessarily, any point with some fixed number of zero coordinates cannot be moved to an element with a different number of zeros. Consequently, when dealing with forms of higher degree, one cannot usually simplify things by the use of matrices. One has to consider the general case.

There is a vast literature devoted to the problem of computation of the $2$-Pythagoras number of a ring, see  \cite{benoist2017, benoist2020, fernando2021,hoffman1999, krasensky2022, krs2022, kp2023} to mention just a few. Similarly there are recent results concerning Waring problems \cite{adaz2020,bs2023, bs2023-2, demiroglu2019, kishore2022, msw2023, pliego2021,pv2021, wooley2016} and rings of $k$-regulous functions \cite{baneckik2023} as well as henselian local rings \cite{mk2022}. On the other hand, there are very few papers dealing with sums of higher powers in real rings, see \cite{becker1982, clpr1996, grimm2015, schmid1994} for higher Pythagoras numbers of fields and \cite{cr1990, ruiz1985, fq2004} for algebroid curves. Beside those, it is known that $p_{2s}(\Z[x])=\infty $ \cite{cldr1982}.

If one is interested in the problem of finiteness of the Pythagoras numbers, then in order to obtain nontrivial problems one has to (usually) consider real rings i.e. rings where $-1$ is not a sum of squares. Fix a positive integer $s$. If $-1$ is a sum of $s$th powers in $R$ then one can use the following identity (cf. \cite[Introduction]{mk2022})
$$s!x=\sum_{r=0}^{s-1}(-1)^{s-1-r} {s-1 \choose r} [(x+r)^s -r^s ]. $$
If $s!$ is invertible in $R$ then we get an explicit bound on the $s$th Pythagoras number
$p_s(R) \leq sG(s)(s_s(R)+1).$
Here, $s_s(R)$ is the $s$th level of $R$ i.e. smallest number of $s$th powers which sum up to $-1$ and $G(s)$ is the big gee functions related to the classical Waring problem. It is finite and explicit bounds are known. This motivates the study of sums of even powers in real rings.

 Let $A$ be a commutative ring with identity, $s\geq1$ be a positive integer and $a \in A$.
\begin{dfn}
We define the $s$-length of $a$ as the smallest positive integer $\ell$ such that $a$ can be written as a sum of $\ell$ $s$-th powers of elements of $A$.
\end{dfn}

For any $n,s,d\in\N$ we consider the cone $\Sigma_{n,sd}^s\subset\R[x_1,\dots,x_n]_{sd}$ of sums of $s$-th powers of forms of degree $d$. If $s$ is even this is a closed, pointed and full-dimensional convex cone. By definition, any form $f$ in $\Sigma_{n,sd}^s$ has a representation $f=\sum_{i=1}^r f_i^s$ for some $r\in\N$ and $f_i\in\R[x_1,\dots,x_n]_d$. The smallest such $r$ is called the $s$-length of $f$.

\begin{dfn}[homogeneous Pythagoras number]
We define the $s$-Pythagoras number $p_s(n,sd)$ of $\R[x_1,\dots,x_n]_{sd}$ as the minimal number $\ell$ such that any $f\in\Sigma_{n,sd}^s$ has such a representation of length at most $\ell$. 
\end{dfn}

The goal of this paper is to study the higher Pythagoras numbers of the polynomial rings. In particular we prove that $p_{2s}(K[x_1,x_2,\dots,x_n])=\infty$ provided that $K$ is a real field and $n>1$. This almost solves an open problem from the paper \cite[Problem 8]{cldr1982}. Cases that are still open are precisely $n=1$ and $s\geq 2$. From the above, we are able to deduce that for a large class of rings, all of its higher Pythagoras numbers are infinite. The second part of the paper focuses on the study of the 4-Pythagoras number. 
It follows from the first part, that similarly as in the quadratic case,  the homogenous $2s$-Pythagoras numbers in polynomial rings tends to infinity if we increase the degree as long as we consider forms in at least 3 variables. This led us to study binary forms, especially fourth powers, as this is the first non quadratic case. For any fixed degree $4d $ there exists a full-dimensional semi-algebraic subset of $\Sigma_{2,4d}^4$ such that any form in it has 4-length at most 4. If the degree is at least 10, then this subset consists of elements of precisely 4-length 4. This is also the smallest length that appears in a full-dimensional subset. Writing binary forms as sums of powers of linear forms has already been done by Reznick.

%Since we want to study the 4-Pythagoras number, we give a different proof of the fact that $\py_4(2,4)=3$. This has already been proven by Reznick by studying the dual cone of $\Sigma_{2,4}^4$. We use a proof more similar to Hilbert's and to Scheiderer's generalized version \cite[§4]{scheiderer2017} . Moreover, we will see that every form on the boundary has length at most 2 and every form of length at most 2 is on the boundary. In particular, the boundary is the union of two $\gl_2(\R)$ orbits.

The problem of decomposition of a binary form as a sum of powers of linear forms goes all the way back to the Sylvester. Here one usually does not restrict to the case of even exponent. Let us mention the paper \cite{co2012} where the binary forms of degree 4 and 5 were studied. In particular, it is shown there, that contrary to the complex case, one can have two different typical ranks of binary forms. The problem of finding all possible real typical ranks was solved (for binary forms) in \cite{blekherman2015}. Let us also mention the papers \cite{adcz2021, bm2021, cr2011, cs2011}. Additionally, we have to mention the moment problems. With this, there is a notion of a Caratheodory number of a closed convex cone. For the cone of sums of $2s$th powers this is the same as the $2s$-Pythagoras number. In particular, moment problem provides a connection between various branches of mathematics \cite{didio2018}. For the recent papers relevant to the topic see \cite{ didio2018, dik2021, dis2018, dis20182, dis2022, fn2010,  fialkow2016, hn2012,  il2017, nie2014, schmudgen2017}.

Our main results concern binary octics. We determine the algebraic boundary of the cone $\Sigma_{2,8}^4$ by studying its dual cone which can be seen as a linear slice of the cone $\Sigma_{3,4}^2$ of ternary quartics that are sums of squares, or equivalently are positive semidefinite (psd). Surprisingly, it turns out that the boundary is irreducible and a generic point of the boundary has a 'special' length 3 representation as a sum of fourth powers of quadratics. Moreover, this is the unique representation as a sum of fourth powers. This allows us to deduce that the 4-Pythagoras number is either 3 or 4.
We are then able to disprove a conjecture of Reznick \cite[Conjecture 7.1]{reznick2011} expecting that sums of fourth powers are ``doubly positive" meaning they can be written as sums of squares of psd forms.

The structure of the paper is as follows. \cref{sec:Higher Pythagoras numbers} contains the study of the higher Pythagoras number of a variety of rings. The main result of this section is the computation of the higher even Pythagoras numbers for the polynomial ring, i.e, $p_{2s}(K[x_1,x_2,\dots,x_n])=\infty$. In \cref{sec:Binary forms} we examine the structure of the cone $\bin$ of binary octics that are sums of fourth powers of quadratics as well as their $4$-Pythagoras number. In the next section, we determine the facial structure of the aforementioned cone and disprove Reznick's Conjecture. We finish the paper with some open problems.

\begin{rem}
In this paper we consider only even powers of homogeneous polynomials. This is because of the following observation. Let $f$ be a homogeneous polynomial of degree $sd$ and consider the $s$-length of $f$. If $s$ is even, then the $s$-length of $f$ in the ring $\R[x_1,\dots,x_n]$ will be the same as in the vector space  $\R[x_1,\dots, x_n]_{sd}$. However, if $s$ is odd, then those lengths will almost surely be different, since if $f$ is a sum of odd powers of polynomials $f_i$, then the $f_i$'s do not have to be homogeneous.

\end{rem}

\begin{rem}
There is a relation between the $2s$-Pythagoras number and its homogeneous counterpart. Of course, the vector space of polynomials in $n-1$ variables with degree at most $2sd$ can be seen, after homogenization as $\R[x_1,\dots, x_n]_{2sd}$. In particular, the $2d$ length of a homogeneous form $f$ is the same as its dehomogenization (cf. \cite[Lemma 6.5]{cldr1982} for the case $s=1$). Hence, we obtain the following equality
$$ p_{2s}(\R[x_1,\dots, x_{n-1}]) = \limsup_{d\rightarrow +\infty} p_{2s}(n, 2sd).$$
\end{rem}

%\textcolor{red}{we need to fix, what does n,s and d refer to. how about 2n as a power, s as a number of variables?}

Elements of the vector space $\R[x_1,\dots, x_n]_{sd}$ will be called $n$-ary $sd$-ics.

\section{Higher Pythagoras numbers}
\label{sec:Higher Pythagoras numbers}
In this section we discuss higher Pythagoras numbers of rings. We are mainly concerned with the polynomial rings over a real field.
In particular, $\py_2(n,2d)$ corresponds to shortest sums of squares representations of forms. Even these are not well-understood in general (see \cite{scheiderer2017}).
One of the main difficulties in studying higher Pythagoras numbers is that we are no longer considering quadratic forms which generally behave much nicer than higher tensor powers.

Firstly, let us show a lower bound for $p_{2s}(n,2sd)$ by counting dimensions. The dimension of  $\R[x_1,\dots, x_n]_{2sd}$ is equal to $n+2sd-1 \choose 2sd$ and this is also the dimension of $\Sigma^{2s}_{n,2sd}$ as this cone contains a generating set $\{l^{2sd}\colon l\in\R[x_1,\dots,x_n]_1\}$ (see \cite[page 2]{reznick1992}) of the real vector space $\R[x_1,\dots, x_n]_{2sd}$. The $2s$-th Pythagoras number $p_{2s}(n, 2sd)$ has to satisfy 
$$p_{2s}(n,2sd){n+d-1 \choose d} \geq {n+2sd-1 \choose 2sd}.$$
In other words, 
$$p_{2s}(n,2sd) \geq \frac{{n+2sd-1 \choose 2sd}}{{n+d-1 \choose d}}.$$
A simple computation shows that if $d$ tends to infinity, then the right hand side tends to a constant value $(2s)^{n-1}$. Upon dehomogenization, this can be translated into the inequality
$$p_{2s}(\R[x_1,\dots, x_{n-1}])\geq (2s)^{n-1}.$$
In particular $p_4(\R[x])\geq 4$.
It is not a priori clear that the numbers $p_{2s}(n,2sd)$ are finite. However, Caratheodory Theorem \cite[Proposition 2.3]{reznick1992}
 gives upper bound of the form
$$p_{2s}(n,2sd) \leq  {n+2sd-1 \choose 2sd}.$$

Our first result is a generalization of \cite[Proposition 4.5']{cldr1982} 
to higher even powers. Let $s>1$ be a positive integer and $K$ be a formally real field. 

\begin{prop}\label{length G}
Let $g(x,y)\in K[x,y]$ be a polynomial of $2s$-length $k$ and $r$ be a positive integer such that $\deg_xg < 2sr$. Then the polynomial $G(x,y):=g(x,y)(y-x^r)^{2s}+1$ has $2s$-length $k+1$.

\begin{proof}
Assume by contrary that this is not the case, i.e. 
\begin{equation}
\label{eq:Gxy}
G(x,y)=\sum_{i=1}^k g_i^{2s}.
\end{equation}
By setting $y=x^r$ we obtain that for each $i$, $g_i=a_i+(y-x^r)h_i$, where $a_i \in K$ and $h_i \in K[x,y]$. We may now rewrite \cref{eq:Gxy} as 
$$g(x,y)(y-x^r)^{2s}+1=\sum_{i=1}^k a_i^{2s}+(y-x^r) {2s \choose 1} \sum_{i=1}^k a_i^{2s-1}h_i + (y-x^r)^2 {2s \choose 2}  \sum_{i=1}^k a_i^{2s-2}h_i^2+\dots +$$
$$+ (y-x^r)^{2s}\sum_{i=1}^k h_i^{2s}.$$
Since $\sum_{i=1}^k a_i^{2s}=1$, both $1$'s cancel out. As a next step, after dividing by $(y-x^r)$, we see that $\sum_{i=1}^k a_i^{2s-1}h_i$ is divisible by $(y-x^r)$. By degree considerations, we see that $\sum_{i=1}^k a_i^{2s-1}h_i = 0$. If not, then the degree of some $h_i$ would be at least $r$, hence the degree of the right-hand side of \cref{eq:Gxy} would be at least $4sr$, and this is larger than the degree of the left-hand side.
Consequently, the above expression can be reduced to
$$g(x,y)(y-x^r)^{2s-2}= {2s \choose 2}  \sum_{i=1}^k a_i^{2s-2}h_i^2+\dots + (y-x^r)^{2s-2}\sum_{i=1}^k h_i^{2s}.$$
We may look at the polynomial $\sum_{i=1}^k a_i^{2s-2}h_i^2$. If this polynomial is nonzero, then at least one of the $h_i$ is divisible by $(y-x^r)$ and we get a contradiction as previously.
As a result, we get that 
$$\sum_{i=1}^k a_i^{2s-2}h_i^2 = 0$$
and for each $i=1,2,\dots, k$, necessarily $a_ih_i=0$, as we are working over a formally real field. In particular, for each $i=1,2, \dots, k$, $a_i=0$ or $h_i=0$. Denote by $l$ the number of nonzero $h_i$ polynomials. Of course, $l<k$. After enumerating, we can rewrite \cref{eq:Gxy} as 
$$g(x,y)(y-x^r)^{2s}+1=(y-x^r)^{2s}\sum_{i=1}^lh_i^{2s}+1$$
where we merged all nonzero $a_i$'s into $1$. As the final step we get 
$$(y-x^r)^{2s}(g(x,y)-\sum_{i=1}^lh_i^{2s})= 0,$$ but this implies that the $2s$-length of $g(x,y)$ is less than $k$, a contradiction.
\end{proof}

\end{prop}

\begin{thm}\label{p_2s=infty}
Let $n>1$ and $s>1$ be positive integers. Then
$p_{2s}(k[x_1,x_2,\dots,x_n])=\infty$
for any formally real field $k$.
\begin{proof}
It is enough to treat the case $n=2$. By the above Proposition, starting from $G(x,y)=1$ we can inductively construct a sequence of polynomials of arbitrarily large length. This finishes the proof.
\end{proof}

\begin{thm}Let $s$ be a positive integer and $n>1$. Then
   \[
p_{2s}(A[x_1,x_2,\dots,x_n])=\infty
    \]
    for any commutative ring in which $-1$ is not a sum of squares.
\end{thm}
    \begin{proof}
One can repeat proof of \cite[Corollary 4.18]{cldr1982} 
verbatim.

\end{proof}

\end{thm}

By a similar use of the notion of the homogeneous Pythagoras number as in \cite{cldr1982} we may easily deduce the following
\begin{cor}
Let $R$ be a regular local ring of dimension at least 3 with formally real residue field. Then $p_{2s}(R)=\infty$ for all $s\geq 1$.
\end{cor}
%\textcolor{red}{dopisać OO(R) i $p_{2d}$}
The above corollary allows us to deduce what are the $2s$-Pythagoras numbers for the rings of regular functions on an algebraic set of dimension at least 3.

Let  $X\subset \R^n$ be an algebraic set of dimension at least 3. Denote by $\OO(X)$ the ring of regular functions on $X$ \cite[Definition 3.2.1]{bcr1998}.
\begin{cor}
If $X$ is a real algebraic set of dimension at least 3, then for any positive integer $s\geq 1$ we have
$p_{2s}(\OO(X))=\infty$.
\end{cor}
\begin{proof}
Let $y$ be a regular point of $X$ and $\mathfrak{m}$ be its maximal ideal. Then the basic properties of the Pythagoras number imply that $p_{2s}(\OO(X))\geq p_{2s}(\OO(X)_{\mathfrak{m}})$. However, the ring on the right hand side is a regular local ring of dimension at least 3, hence $p_{2s}(\OO(X)_{\mathfrak{m}})=\infty$ and thus $p_{2s}(\OO(X))=\infty$.
\end{proof}

We might also get an insight into the behavior of higher Pythagoras numbers of $k$-regulous functions (\cite[Definition 1.1]{baneckik2023}).  $k$-regulous functions on $\R^n$ are rational functions with a continuous extension to $\R^n$, such that the extension is of class $\mathcal{C}^k$. Despite rather simple definition, they posses many remarkable properties (cf. \cite{fhmm2016}).  Denote by $\mathcal{R}^k(\R^n)$ the ring of $k$-regulous functions on $\R^n$ for $k>0$ and $n>2$. Theorem \ref{p_2s=infty} and \cite[Theorem 3.7]{baneckik2023} readily imply

\begin{cor}
For fixed positive integers $s>0$ and $n>2$ we have
 $$ \limsup_{k\rightarrow \infty}p_{2s}(\mathcal{R}^k(\R^n)))=\infty.$$

\end{cor}
It is worth mentioning that it is still not known whether the above quantities are finite or not.

\begin{cor}\label{cor 2.7}
Let $n\ge 3$ and $s\in\N$. Then $\py_{2s}(n,2sd)\to\infty$ as $d\to\infty$.
\end{cor}

This proof does not work for binary forms. We do not know if the $4$-Pythagoras number is bounded in this case.

In \cite{bk2023} the first author and Błachut showed that the $2$-Pythagoras number of coordinate rings of a large class of surfaces $\{z^2-f(x,y)=0\} \subset \R^3$ is infinite. Here we are able to generalize this result to the $2s$-Pythagoras numbers.

\begin{dfn}
Let $f(x,y) \in \R[x,y]$ be a non-constant polynomial such that $f(0,0)=0$. Let $d = \deg_y f(x,y)$ and $\alpha x^b y^d$ be a monomial with the largest possible $b$ and nonzero coefficient $\alpha$ among all monomials in $f(x,y)$.
We say that $f(x,y)$ is strictly admissible if:
\begin{itemize}
\item[a)]$\alpha>0$ and $b,d$ are even, or
\item[b)] $b$ or $d$ is odd.
\end{itemize}
We say that $f(x,y)$ is an admissible polynomial if there exists an invertible matrix $M\in \gl_2(\R)$ such that $f\circ M$ is a strictly admissible polynomial.
\end{dfn}
See \cite{bk2023} for examples and further properties of admissible polynomials. In particular, an admissible polynomial has to be positive on an open unbounded set. We may now prove the following.
\begin{thm}
Let $f(x,y)$ be an admissible polynomial which is not a square. Then 
\[
\py_{2s}(\R[x,y,\sqrt{f(x,y)}])=\infty.
\]
\end{thm}

\begin{proof}
We will carry out a proof for $s=2$, as this is the first non-quadratic case. The proof for higher powers works similarly.

Clearly, for any invertible matrix $M\in GL_2(\R)$, the rings $\R[x,y,\sqrt{(f\circ M)(x,y)}]$ and $\R[x,y,\sqrt{f(x,y)}]$ are isomorphic, hence without loss of generality we may assume that $f(x,y)$ is a strictly admissible polynomial.

We begin the proof by constructing a sequence of polynomials (cf. \cite[Definition 1.10]{bk2023}).

Consider a sequence of positive integers $(r_m)_{m\geq 1}$ such that
\begin{itemize}
\item $r_1>\deg f$ and $r_i> 4\sum_{k=1}^{i-1}r_{k}$
\item $f(x,x^{r_i})$ is a polynomial in $x$, which is either of odd degree, or it has even degree, and the leading coefficient is strictly positive. 
\end{itemize}

We define a sequence of polynomials in the following way:
\begin{itemize}
\item $F_1=1$
\item $F_n=F_{n-1} \cdot (y-x^{r_{n-1}})^4+1$

\end{itemize}
 Fix a sequence of polynomial $(F_m)_{m\geq 1}$ satisfying the above properties. By the Proposition $\ref{length G}$ $4$-length of $F_m$ in $\R[x,y]$ is equal to $m$.

Assume by contrary that the $4$-Pythagoras number of $\R[x,y,\sqrt{f(x,y)}]$ is finite and equal to $L$. Consider the polynomial $F_{L+1}$. By assumption, we have the following 
\begin{equation}\label{eq0}
F_{L+1}=\sum_{i=1}^L (f_{i,1} + \sqrt{f}g_{i,1})^4
\end{equation}
which translates into the equation
\begin{equation}\label{eq1}
F_{L+1}=\sum_{i=1}^Lf_{i,1}^4 +6f\sum_{i=1}^Lf_{i,1}^2g_{i,1}^2+f^2\sum_{i=1}^Lg_{i,1}^4
\end{equation}
as the part that is divisible by $\sqrt{f}$ has to vanish.
Let us now substitute $y=x^{r_L}$. We assumed that $f(x,x^{r_L})$ is either of odd degree or of even degree with positive leading coefficient and does not vanish identically. Also, $f(x,x^{r_L})$ is divisible by $x$. This readily implies that
$f_{i,1}=a_i+(y-x^{r_L})f_{i,2}$ for some $a_i \in \R$ such that $\sum_{i=1}^L a_i^4=1$ and additionally $g_{i,1}=(y-x^{r_L})g_{i,2}$.
We may now rewrite  \cref{eq1} as 
$$F_L(y-x^{r_L})^4+1=\sum_{i=1}^La_i^4+
4(y-x^{r_L})\sum_{i=1}^La_i^3f_{i,2}+
6(y-x^{r_L})^2\sum_{i=1}^La_i^2f_{i,2}^2+
4(y-x^{r_L})^3\sum_{i=1}^La_if_{i,2}^3+$$
$$(y-x^{r_L})^4\sum_{i=1}^Lf_{i,2}^4+6f(y-x^{r_L})^2\sum_{i=1}^L(a_i+(y-x^{r_L})f_{i,2})^2g_{i,2}^2+f^2(y-x^{r_L})^4\sum_{i=1}^Lg_{i,2}^4.$$
After cancelling 1's and dividing by $(y-x^{r_L})$ we see that the degree considerations forces $\sum_{i=1}^La_i^3f_{i,2}=0$.
Again, by dividing by  $(y-x^{r_L})$ we get that
$$(y-x^{r_L}) | \sum_{i=1}^La_i^2f_{i,2}^2+6f\sum_{i=1}^La_i^2g_{i,2}^2.$$

If the polynomial on the right is nonzero, then necessarily $\sum_{i=1}^Lf_{i,2}^4$ or $f^2\sum_{i=1}^Lg_{i,2}^4$ have to be of degree at least $2r_L$. Furthermore, the right-hand-side of \cref{eq1} will be of degree at least $6r_L$, while the left-hand side is of degree $4\sum_{i=1}^L r_i$. The contradiction follows from the assumption on the $r_i's$.
As a result $ \sum_{i=1}^La_i^2f_{i,2}^2+6f\sum_{i=1}^La_i^2g_{i,2}^2$ is the zero polynomial. Even more, we know that for each $i=1,2,\dots, L$ $a_if_{i,2}= 0$ and $a_ig_{i,2}= 0$, because if not, then $f$ would attain only nonpositive values, but this is a contradiction with the assumption of admissibility.
Furthermore, there exists at least one $i$ such that $a_i\neq0$ and for such, we have $f_{i,2}=g_{i,2}=0.$ By merging all of the nonzero $a_i's$ into one, renumbering and adding artificial zero polynomials if necessary, we can rewrite \cref{eq0} as
$$F_L(y-x^{r_L})^4+1=(y-x^{r_L})^4\sum_{i=1}^{L-1}(f_{i,2} + \sqrt{f}g_{i,2})^4 +1$$ 
which is equivalent to 
$$F_L=\sum_{i=1}^{L-1} (f_{i,2} + \sqrt{f}g_{i,2})^4.$$
In this process, we managed to reduce the number of summands in \cref{eq0} as well as reduce the polynomial $F_{L+1}$ to $F_L$.

After repeating the above procedure $L-2$ times we get the following
$$F_2=(y-x^{r_1})^4+1=f_{1,L}^4+ff_{1,L}^2g_{1,L}^2+f^2g_{1,L}^4.$$
After repeating the above reasoning one last time, we get that $f_{1,L}=1$ and $g_{1,L}=0$. This implies an equality of polynomials of the type $(y-x^{r_1})^4=0$ which is absurd, hence the proof is done.
\end{proof}

\subsection{Odd powers and signed Waring problem}
In the previous section we have considered only sums of even powers of real functions. Consider now a similar problem, but this time let us allow sums and differences of polynomials.

\begin{dfn}
Define the $s$-Waring number, $w_s(n,sd)$ of $\R[x_1,\dots,x_n]_{sd}$ as the minimal number $\ell$ such that any $f\in\R[x_1,\dots,x_n]_{sd}$ has a presentation as a sum $\sum_{i=1}^\ell \epsilon f_i^d$, where $f_i$ is a homogeneous polynomial of degree $d$ and $\epsilon_i \in \{-1,1 \}$.
\end{dfn}
In the above definition we can require that all homogeneous polynomials of degree $sd$ can be written in that form since the generating set of $\R[x_1,\dots,x_n]_{sd}$ as a vector spaces consists of $sd$ powers of linear forms, hence, the value of the $s$-Waring number is at most the dimension of $\R[x_1,\dots,x_n]_{sd}$ as a real vector space. Although the definition of $s$-Waring number is very similar to the $s$-Pythagoras numbers, the behavior turns out to be very different. For the next result see \cite[Section 3.3]{bt2015} and \cite{fos2012}

\begin{prop}
The following inequality holds
$$ w_s(n,sd) \leq 2 s^{n-1}.$$
\end{prop}
In particular, for a fixed $s, n$ we know that $w_s(n,sd)$ are uniformly bounded. This is in sharp contrast to Corollary \ref{cor 2.7}. The problem of computing the Waring rank of a complex form has been studied for a long time, see \cite{bimm2021, ds2025, flavi2025, hm2022} for recent papers. In particular, the main difference between considering sums of even powers of real forms is that $-1$ is not a sum of even powers. This implies that the cone of sums of even powers is a proper subset $\R[x_1,\dots,x_n]_{sd}$, contrary to what happens in the complex case.

\section{Binary octics}
\label{sec:Binary forms}
For the rest of the paper we will work with binary forms. 
We do not know if the higher homogenous Pythagoras numbers increase if the degree tends to infinity. This is of course not the case for the 2-Pythagoras number as any positive semidefinite binary form can be written as a sum of at most two squares. 
The cone $\Sigma_{2,4}^4$ of binary quartics that are sums of fourth powers of linear forms is well understood by \cite{reznick1992}. We therefore consider the case of octics and prove that the $4$-Pythagoras number is either 3 or 4 in this case (\cref{thm:bound_py4}). Moreover, we show \cref{thm:algebraic_boundary} which states that the algebraic boundary of the cone $\bin$ is irreducible.

Throughout the next two sections length always means $4$-length.
We will often use the sum of $4$-th powers map
    \[
    \phi_d^k\colon \Rx_d^k\to \Rx_{4d},\quad (p_1,\dots,p_k)\mapsto \sum_{i=1}^k p_i^4
    \]
    and its differential at the point $\ul p=(p_1,\dots,p_k)$
    \[
    \dphi_d^k(\ul p)\colon \Rx_d^k\to \Rx_{4d},\quad (q_1,\dots,q_k)\mapsto \sum_{i=1}^k p_i^3 q_i.
    \]
The image of the differential is given by $\id{p_1^{3},\dots,p_k^{3}}_{4d}$.

Moreover, for $f_1,\dots,f_k\in\C[x_1,\dots,x_n]_d$ we denote by $\vc(f_1,\dots,f_k)\subset \P^{n-1}=\P^{n-1}(\C)$ the zero set of the forms $f_1,\dots,f_r$, i.e. the set of all points $\xi\in\P^{n-1}$ such that $f_i(\xi)=0$ for all $i=1,\dots,k$.

The following simple proposition shows that one could hope for a constant 4-Pythagoras number of binary forms and possibly prove $p_4(2,4d)= 4$. However, this set being full-dimensional is a very small indication. We give a direct, slightly combinatorial argument here. 
One may also use powers of linear forms instead of monomials and use known cases of Fr\"oberg's conjecture (see \cite[Lemma 2.2, Theorem 2.3]{lors2019}).

\begin{prop}
Let $n=2$. For every $d\ge 1$ the \sa\ set of all forms in $\Sigma^4_{2,4d}$ of length at most 4 is a full-dimensional \sa\ set.
\end{prop}
\begin{proof}
It suffices to show that the differential $\dphi^4_d(\ul p)$ is surjective for one specific (and therefore a generic) choice of $\ul p\in\Rx^4_d$.
Let $r\in\{0,1,2\}$ such that $d\equiv r\mod 3$.
Choose $p_1=x^d, p_2=y^d$ and $p_3,p_4$ depending on $r$. In every case we show that every monomial in $\Rx_{4d}$ is contained in $I_{4d}=\id{p_1^3,\dots,p_4^3}_{4d}$ which is the component of degree $4d$ of the ideal $I=\id{p_1^3,\dots,p_4^3}$.

In every case, since $x^{3d},y^{3d}\in I$ we see that for $k\le d$ and $k\ge 3d$ the monomial $x^ky^{4d-k}$ is contained in $I$.

\ul{$r=0\,$}: $p_3=(x^2y)^\frac{d}{3},\, p_4=(xy^2)^\frac{d}{3}$

We have $p_3^3=x^{2d}y^d,\, p_4^3=x^dy^{2d}$ which implies that $x^ky^{4d-k}\in I$ for $2d\le k\le 3d$ and $d\le k\le 2d$. Hence $I_{4d}=\Rx_{4d}$.

\ul{$r=2\,$}: $p_3=(x^2y)^\frac{d-2}{3}xy,\, p_4=(xy^2)^\frac{d-2}{3}xy$

We have $p_3^3=x^{2d-1}y^{d+1},\, p_4^3=x^{d+1}y^{2d-1}$ which shows $x^ky^{4d-k}\in I$ for $2d-1\le k\le 3d-1$ and $d+1\le k\le 2d+1$. Hence $I_{4d}=\Rx_{4d}$.

\ul{$r=1\,$}: In this case monomials do not suffice. Consider $p_3=x^\frac{d+2}{3}y^\frac{2d-2}{3},\, p_4=x^d+y^d$.

Since $p_3^3=x^{d+2}y^{2d-2}$, the monomial $x^ky^{4d-k}$ is in $I$ if $d+2\le k\le 2d+2$. We are therefore missing $k=d+1$ and $2d+3\le k\le 3d-1$.
We have
\[
p_3^3=x^{3d}+y^{3d}+3x^{2d}y^d+3x^dy^{2d}\equiv x^{2d}y^d+x^dy^{2d} \mod I_{3d}
\]
and therefore $(x^{2d}y^d+x^dy^{2d})xy^{d-1}=x^{2d+1}y^{2d-1}+x^{d+1}y^{3d-1}\equiv x^{d+1}y^{3d-1} \mod I_{4d}$. Lastly for every $0\le s\le d$ we see
\[
(x^{2d}y^d+x^dy^{2d})x^sy^{d-s}=x^{2d+s}y^{2d-s}+x^{d+s}y^{3d-s}\equiv x^{2d+s}y^{2d-s} \mod I_{4d}
\]
since $d+s\in\{d,\dots,2d\}$.
\end{proof}

We study the cone of binary octics that are sums of fourth powers of quadratic forms. Via the Veronese map and duality this is equivalent to studying a certain intersection of the psd/sos cone of ternary quartics. It can also be seen as the cone of sums of fourth powers of linear forms in $\R[a,b,c]/(b^2-ac)$, the coordinate ring of the second Veronese of $\P^1$.

We show that the algebraic boundary $\partial_a\bin$ (i.e. the Zariski closure of the boundary of $\bin$) is an irreducible hypersurface and consists of binary octics of length at most 3. Using this we deduce that every binary octic that is a sum of fourth powers of quadratic forms can be written as a sum of at most four such forms.\\

The first goal is to show that any element on the boundary of $\bin$ has 4-length at most 3 (\cref{cor:boundary_length}).
\subsection{The dual cone}
For any $n\in\N$ consider the \todfn{apolarity pairing} on $\R[x_1,\dots,x_n]$ and on $\C\otimes \R[x_1,\dots,x_n]$. Given two monomials $\ul x^\alpha,\, \ul x^\beta$, $\alpha,\beta\in\Z_{\ge 0}^n$ we define
\[
\bil{\ul x^\alpha}{\ul x^\beta}=\begin{cases}\frac{\beta!}{\alpha!}\ul x^{\beta-\alpha}\quad &,\text{if } \alpha\le \beta\\ 0\quad &,\text{else} \end{cases}
\]
where $\alpha!=\prod_{i=1}^n (\alpha_i!)$.
This is the same as having the first element act on the second via differentiation. For any $d\in\N$ this gives a scalar product on $\R[x_1,\dots,x_n]_d$. Over the complex numbers this is not true, but it is still a perfect pairing.

Consider the dual cone 
\[
(\bin)^\star=\{L\in\R[x,y]_8^\du\colon L(q^4)\ge 0\, \all q\in\R[x,y]_2\}.
\]
Let $q(a,b,c)=ax^2+bxy+cy^2$ with new variables $a,b,c$ and let $L=\sum_{i=0}^8 a_i x^iy^{8-i}\in\R[x,y]_8$. By identifying $\R[x,y]_8^\du$ and $\R[x,y]_8$ via the apolarity pairing we  get
\begin{equation}
\label{eqn:eval}
\bil{L}{q(a,b,c)^4}=\sum_{i=0}^8 a_i \bil{x^iy^{8-i}}{x^iy^{8-i}} g_i(a,b,c)
\end{equation}
for some forms $g_i\in\R[a,b,c]_4$. Hence, $L\in(\bin)^\star \iff \bil{L}{q(a,b,c)^4}\in\ter\subset\R[a,b,c]_4$. It was proven by Hilbert that the cone of positive semidefinite ternary quartics is equal to the sum of squares cone. 
This shows that the dual cone $(\bin)^\star$ is a linear slice of the sos cone $\ter$ of ternary quartics (i.e. it is an intersection of $\ter$ with a linear subspace). From \cref{eqn:eval} we see that the coefficients of the ternary quartic $\bil{L}{q(a,b,c)^4}$ are not linearly independent but that the ternary quartics are contained in a 9-dimensional subspace.

For $F\in\R[a,b,c]_4$ write $F=\sum_{\alpha\in\Z^3} c_\alpha \ul a^\alpha$, where $\ul a = (a,b,c)$. We calculate that the subspace $U$ of $\R[a,b,c]_4$ spanned by ternary quartics of the above form is given by the vanishing of the six linear forms
\begin{align*}
\frac{2}{3}c_{220}-c_{301},\, 3c_{130}-c_{211},\, \frac{2}{3}c_{022}-c_{103},\, 3c_{031}-c_{112},\, 12c_{040}-c_{121},\, 6c_{040}-c_{202}.
\end{align*}
Equivalently, it is the orthogonal complement of $(b^2-ac)_4$ in $\R[a,b,c]_4$ with respect to apolarity.
Understanding $(\bin)^\star$ is the same as studying $\ter\cap U$. The facial structure of $\ter$ has been studied by Kunert in \cite{kunert2014} and is rather explicitly known.

Let $q_1,\dots,q_s\in \R[x,y]_2$ and assume $f=\sum_{i=1}^s q_i^4$ is on the boundary of $\bin$. Let $L\in(\bin)^\star$ be a linear functional in $\R[x,y]_8^\du$ vanishing at $f$, then $L(f)=0$ and thus $L(q_i^4)=0$ ($i=1,\dots,s$) which shows that the points in $\P^2(\R)=\P(\R[x,y]_2)$ corresponding to $q_1,\dots,q_s$ are zeros of the psd ternary quartic $F=\bil{L}{q(a,b,c)^4}$.

%We want to show that there is no psd ternary quartic in $U$ vanishing at four distinct points.

The $\gl_2(\R)$ action on $\R[x,y]$ gives rise to an embedding of $\gl_2(\R)$ into $\gl_3(\R[a,b,c]_1)$ which allows some automorphisms on $\R[a,b,c]_4$ that fix the subspace $U$ as well as $\ter$.

Since we are dealing with real zeros of psd elements understanding the behaviour of derivatives is essential. Let $L\in\R[x,y]_8$ and $u=(u_1,u_2,u_3)\in\R^3$. Via a direct calculation we see
\[
\partial_u\bil{L}{q(a,b,c)^4}=4\bil{L}{q(a,b,c)^3(u_1x^2+u_2xy+u_3y^2)}
\]
where the partial derivatives are taken with respect to the variables $a,b,c$.

A different way to arrive at the same point is as follows. Let $v_2\colon\P^1\to\P^2$ be the second Veronese map. With coordinates $a,b,c$ on $\P^2$ the image is given as the vanishing set of $b^2-ac$. This gives rise to an isomorphism $\R[x,y]_{2d}\cong\R[a,b,c]/(b^2-ac)_d$ for any $d\ge 1$. Moreover, under this isomorphism the cones $\bin$ and $\Sigma^4\R[a,b,c]/(b^2-ac)_1$ are isomorphic. Here the second cone consists of elements that are sums of fourth powers of linear forms. Dualizing the second cone we get
\[
(\Sigma^4\R[a,b,c]/(b^2-ac)_4)^\star\cong \ter\cap (b^2-ac)_4^\perp.
\]
The fact that the cone of sums of squares $\ter$ and the cone $P_{3,4}$ of psd ternary quartics coincide will be essential for certain calculations as testing for non-negativity is very hard.

Let $L\in\partial(\bin)^\star$, then $F:=\bil{L}{q(a,b,c)^4}$ is a psd ternary quartic in $\R[a,b,c]_4$ which therefore is also a sum of squares. Since $L$ is on the boundary there exists $g\in\R[x,y]_2$ such that $L(g^4)=0$. Thus $F$ has a real zero and lies on the boundary of the sos cone $\ter$.

We show (Proposition \ref{prop:infinite_zeros} and Theorem \ref{thm:at_most_3_zeros}) that there are two cases. Either
\begin{enumerate}
\item $|\vc(F)(\R)|\le 3$ or
\item $|\vc(F)(\R)|=\infty$, $(L^\perp)_3\neq \{0\}$ and $F=l^4$ for some $l\in\R[a,b,c]_1$.
\end{enumerate}
We use the characterization of faces of $\ter$ from Kunert \cite{kunert2014}. In this thesis all faces of the sums of squares cone $\ter$ are characterized in terms of the types of singularities of its relative interior points.
In particular, a quartic $F\in\ter$ has either infinitely many real points or at most 4. 
Our main technical result is to prove that there is no quartic in $\ter\cap U$ that has exactly four real zeros.

For the rest of the paper we will slightly abuse notation and will see an element of $(\bin)^\star$ as either an element of $\R[x,y]_8$ or a ternary quartic in $\R[a,b,c]_4$. In general, we use $f$ for binary octics, $F$ for ternary quartics and $L$ for linear functionals in $(\R[x,y]_8)^\du$. As we have seen the latter two are both dual elements to $\R[x,y]_8$, however both points of view have their separate advantages.

\begin{rem}
Our results imply the following.
    There are the following types of ternary quartics $F$ on the boundary of the dual cone $(\bin)^\star$.
    \begin{enumerate}
        \item $F=l^4$ with $l\in\R[a,b,c]_1$.
        \item $\vc(F)(\R)=\{(\xi_1\colon \xi_2\colon \xi_3)\}$ with $\xi_2^2-4\xi_1\xi_3>0$, i.e. the associated binary quadratic form is indefinite. And for every such point $\xi$ there exists $F$ whose only real zero is $\xi$.
        \item $|\vc(F)(\R)|=2,3$ and all points correspond to indefinite binary quadratics. In this case not every pair and triple of points is possible.
    \end{enumerate}
\end{rem}

\subsection{Main results}

For $L\in\R[x,y]_8$ denote by $L^\perp$ the apolar ideal of $L$, i.e. the ideal generated by all forms $g\in\R[x,y]$ such that $\bil{g}{L}=0$. Note that $\R[x,y]/L^\perp$ is an artinian Gorenstein ring with socle in degree 8. Moreover, since we are working with binary forms the ideal $L^\perp$ is also a complete intersection with the degrees of the two generators summing up to 10. Considering this Gorenstein ideal already proved to be a main ingredient in studying sum of squares cones, for example in \cite{blekherman2013}.

\begin{prop}
\label{prop:infinite_zeros}
Let $L\in(\bin)^\star$ and $F=\bil{L}{q(a,b,c)^4}$.
If $|\vc(F)(\R)|=\infty$, then $l'^3\in (L^\perp)_3$ for some $l'\in\R[x,y]_1$ and every element in $(L^\perp)_6$ is divisible by $l'^2$.
Moreover, there exists $l\in\R[a,b,c]_1$  such that $F=l^4$.
\end{prop}
\begin{proof}
By \cite[§2.3.]{kunert2014} either $F=p^2$ for some indefinite, irreducible $p\in\R[a,b,c]_2$ or $F=l^2u$ for some $u\in\R[a,b,c]_2, l\in\R[a,b,c]_1$. Since $F$ is singular at any real zero $u=(u_0:u_1:u_2)\in\P^2(\R)$ we see 
\[
\bil{L}{q(u_0,u_1,u_2)^3h}=0\quad \forall h\in \R[x,y]_2.
\]
If $F=l^2u$, there exists a 2-dimensional subspace $V\subset\R[x,y]_2$ such that for every $q\in V$, $q^3\in L^\perp$, $V$ is given as the orthogonal of $\spn(l)\subset\R[a,b,c]_1\cong \R[x,y]_2$. Since $\spn(q^3\colon q\in V)=V^3$, we see $V^3\R[x,y]_2\subset L^\perp$.

If $V$ is spanned by a regular sequence, the Hilbert function of the ideal generated by $V$ is $(1,2,1)$. Since 
\[V^3\R[x,y]_2=V^2(V\R[x,y]_2)=V^2(\R[x,y]_4)=V(V\R[x,y]_4)=V\R[x,y]_6=\R[x,y]_8
\]
and $\bil{\cdot}{L}$ vanishes on it we have a contradiction.
Therefore $V$ is spanned by $l'l_1,l'l_2$ for some linear forms $l',l_1,l_2\in\R[x,y]_1$. Then $V=l'\R[x,y]_1$ and $V^3\R[x,y]_2=l'^3\R[x,y]_5$, hence $l'^3$ is in the apolar ideal of $L$.
Additionally, since the apolar ideal is a complete intersection and there is one generator in degree at most 3, the other generator has degree at least 7 (see \cite[Thm 1.44.]{ika1999}). In particular, every form in degree 6 is divisible by $l'^2$. Thus every real zero of $F$ is a zero of $l$. Moreover, $l'^8\in L^\perp$.

Therefore $\bil{L}{l'^8}=0$. We claim that this already implies $F=l^4$. After a \cc\ we can assume $l'=x$ and therefore $F$ vanishes at $(1:0:0)$.
We calculate
\[
U\cap I((1:0:0))_2^2=\spn(a b^{3} + 3 a^{2} b c, b^{2} c^{2} + \frac{2}{3} a c^{3}, b^{3} c + 3 a b c^{2}, b^{4} + 12 a b^{2} c + 6 a^{2} c^{2}, c^{4}, b c^{3}).
\]
The Gram matrices of an element 
\[
a b^{3} a_{1} + 3 a^{2} b c a_{1} + b^{2} c^{2} a_{2} + \frac{2}{3} a c^{3} a_{2} + b^{3} c a_{3} + 3 a b c^{2} a_{3} + b^{4} a_{4} + 12 a b^{2} c a_{4} + 6 a^{2} c^{2} a_{4} + c^{4} a_{5} + b c^{3} a_{6}
\]
in the intersection have the form
\[
\left(\begin{array}{rrrrr}
0 & 9 a_{1} & 3 a_{1} & 36 a_{4} - \lambda_{1} & 9 a_{3} - \lambda_{2} \\
9 a_{1} & 36 a_{4} & \lambda_{1} & \lambda_{2} & 2 a_{2} \\
3 a_{1} & \lambda_{1} & 6 a_{4} & 3 a_{3} & \lambda_{3} \\
36 a_{4} - \lambda_{1} & \lambda_{2} & 3 a_{3} & 6 a_{2} - 2 \lambda_{3} & 3 a_{6} \\
9 a_{3} - \lambda_{2} & 2 a_{2} & \lambda_{3} & 3 a_{6} & 6 a_{5}
\end{array}\right)
\]
for some $\lambda_1,\lambda_2,\lambda_3\in\R$ with respect to the basis $ab, ac, b^2, bc, c^2$.
By checking all principle $2\times 2$ minors, it is easy to see that the only form in the intersection having a psd Gram matrix is $a_5c^4$ with $a_5>0$.

Lastly, assume $F=p^2$ with $p$ indefinite, irreducible. Since $F\in U$ the quadratic form $\bil{F}{b^2-ac}\in\R[a,b,c]_2$ has to vanish identically. This gives six equations. One then easily checks on a computer that the determinant of the quadratic form $p$ is contained in the radical of the ideal generated by the six forms above, i.e. $p$ is reducible, a contradiction.
\end{proof}

We now prepare to prove $|\vc(F)(\R)|\le 3$ if $\vc(F)(\R)$ is finite. Recall that we already know from \cite{kunert2014} that in this case $|\vc(F)(\R)|\le 4$.

\begin{lem}
\label{lem:not_ci}
Let $q_1,q_2\in \P(\R[x,y]_2)$ be linearly independent with a common factor and $q_1^3,q_2^3\in L^\perp$. Then $l^3\in(L^\perp)_3$ for some $l\in\R[x,y]_1$.
\end{lem}
\begin{proof}
Let $q_1=ll_1, q_2=ll_2$ with $l,l_1,l_2\in\R[x,y]_1$. By assumption $(ll_1)^3 \R[x,y]_2, (ll_2)^3\R[x,y]_2\subset L^\perp$. Since $l_1,l_2$ are a regular sequence, the Hilbert function of the ideal $\id{l_1^3,l_2^3}$ is $(1,2,3,2,1)$. Thus
\[
L^\perp\supset (ll_1)^3 \R[x,y]_2+ (ll_2)^3\R[x,y]_2=l^3 \id{l_1^3,l_2^3}_5=l^3 \R[x,y]_5.
\]
\end{proof}

\begin{lem}
\label{lem:pd_impossible}
Let $L\in(\bin)^\star$ and $F=\bil{L}{q(a,b,c)^4}$. $F$ does not vanishes at any real point $\xi\in\P^2(\R)$ corresponding to a definite binary quadratic form $q$.
\end{lem}
\begin{proof}
Assume otherwise. After a \cc\ we can assume that $q=x^2+y^2$.
By \cite[Thm 9.5.]{reznick1992} $q^4$ is a sum of five 8th powers of real linear forms.
This shows that $F$ vanishes at least at 5 real points, hence vanishes at infinitely many real points. By \cref{prop:infinite_zeros} there exists $l\in\R[x,y]_1$ such that $l|q$ which is impossible if $q$ is positive definite.
\end{proof}

\begin{thm}
\label{thm:at_most_3_zeros}
Let $F\in\partial(\bin)^\star=\partial\Sigma_{3,4}^2\cap U$ with $\vc(F)(\R)$ finite. Then $|\vc(F)(\R)|\le 3$.
\end{thm}
\begin{proof}
By \cite[Lemma 1.38.]{kunert2014} we have $|\vc(F)(\R)|\le 4$ since $\vc(F)(\R)$ is finite.
Assume $|\vc(F)(\R)|= 4$. Since $\vc(F)(\R)$ is finite, no three of its real zeros are collinear.
Let $q_1,\dots,q_4\in\R[x,y]_2$ be the corresponding quadratic forms. If any two of them were not a complete intersection then by \cref{lem:not_ci} there exists $l\in\R[x,y]_1$ such that $l^3\in (L^\perp)_3$ and thus $(ll')^4\in L^\perp$ for all $l'\in\R[x,y]_1$ which shows that $F$ has infinitely many zeros.

Pick three of these points $\xi_1,\xi_2,\xi_3\in\P^2(\R)$ and let $I\subset \R[a,b,c]$ be the ideal of forms vanishing at $\xi_1,\xi_2,\xi_3$. The degree 2 component $I_2$ is generated by $q_1=l_{12}l_{13}, q_2=l_{13}l_{23}, q_3=l_{13}l_{23}$ where $l_{ij}$ is the line through $\xi_i,\xi_j$. Concretely, $l_{ij}$ is given as
\[
l_{ij}=\det\begin{pmatrix}
\xi_i^1 & \xi_i^2\\ \xi_j^1 &\xi_j^2
\end{pmatrix}a+
\det\begin{pmatrix}
\xi_i^2 & \xi_i^0\\ \xi_j^2 &\xi_j^0
\end{pmatrix}b+
\det\begin{pmatrix}
\xi_i^0 & \xi_i^1\\ \xi_j^0 &\xi_j^1
\end{pmatrix}c
\]
where $\xi_i=(\xi_i^j)_j$.
Consider the $15\times 15$-matrix where the first 9 rows correspond to a basis of $U$ (with respect to some basis of $\R[a,b,c]_4$) and the last 6 correspond to generators of $(I_2^2)_4$. Since $F$ is psd it vanishes to order 2 at $\xi_1,\xi_2,\xi_3$ and is therefore contained in $U\cap (I_2^2)_4$ which shows that the determinant of this matrix is zero.

Since $F\in (I_2^2)_4$ we can represent $F=(q_1,q_2,q_3)G(q_1,q_2,q_3)^t$ with a $3\times 3$-Gram matrix with respect to the basis $(q_1,q_2,q_3)$. Since $F$ is psd (and thus sos since it is a ternary quartic) and vanishes at the real points $\xi_1,\xi_2,\xi_3$, every sos representation of $F$ corresponds to such a matrix $G$, i.e. there is no sos representation $F=\sum_{i=1}^r p_i^2$ where $\spn(p_1,\dots,p_r)\not\subset\spn(q_1,q_2,q_3)$. Since $F$ has a fourth real zero $F$ is a sum of two squares and there exists a psd matrix $G$ of rank 2. Note that $G$ cannot have rank 1 since $\vc(F)(\R)$ is finite.

We want to use \texttt{Mathematica} \cite{mathematica} to prove that this is not possible.
Let $J$ be the ideal generated by the determinant of the $15\times 15$-matrix and by the determinant of the $3\times 3$-matrix $G$. Moreover, let $\chi_G(t)$ be the characteristic polynomial of $G$, $\chi_G(t)=t^3+u_2t^2+u_1t+u_0$. We saturate the ideal $J$ at $u_1$ to get $J_1:=(J\colon u_1^\infty)$. If $u_1$ was 0 then the matrix $G$ would have rank 1 which is impossible. We can now calculate the minimal primes of $J_1$ of which there are 4.

For every minimal prime we use \texttt{FindInstance} in \texttt{Mathematica} to prove that there are no real points $\xi_1,\xi_2,\xi_3$ such that the equations in the prime ideal and the two psd conditions from the coefficients of the characteristic polynomial are satisfied.
\end{proof}

\begin{rem}
\label{rem:proof_cases}
Note that we cannot take the two equations defining $J$ and input them into \texttt{Mathematica} instead of the minimal primes as the calculation does not finish. Similarly, we can only compute the minimal primes after saturating.

To actually perform the calculations we mod out the action of $\gl_2(\R)\subset\gl_3(\R)$ which fixes the subspace $U$. The embedding is given by the $\gl_2(\R)$ action on the second Veronese of $\P^1$ induced by the $\gl_2(\R)$ action on $\P^1$.

By \cref{lem:not_ci} we may assume that any pair of $q_i$'s forms a complete intersection and that the triple is linearly independent. Moreover, by \cref{lem:pd_impossible} no $q_i$ is positive definite.

If all $q_i$ would have rank 1 they are 
\[
(1:0:0), (0:0:1), (1:2:1),
\]
but their third powers generate $\R[x,y]_8$, hence this case cannot occur. Since none of the quadratic forms is positive definite but at least one has rank 2 we can achieve one of the following two forms by using a \cc:
\[
(0:1:0), (1:a:1), (1:b:c)
\]
or
\[
(0:1:0), (1:a:-1), (1:b:c).
\]
Moreover, note that if we make the first point $(1:0:-1)$ in the last two cases the saturation does not terminate any longer.

We thus have to do the calculations above for the two cases
\begin{enumerate}
\item $(0:1:0), (1:a:1), (1:b:c)$,
\item $(0:1:0), (1:a:-1), (1:b:c)$.
\end{enumerate}
\end{rem}

\begin{rem}
The code for the two calculations can be found in the appendix.
\end{rem}

\begin{cor}
\label{cor:boundary_length}
Every point on the boundary $\partial\bin$ has 4-length at most 3.
\end{cor}

\begin{thm}
\label{thm:bound_py4}
The 4-Pythagoras number of binary octics satisfies $\py_4(2,8)\in\{3,4\}$.
\end{thm}
\begin{proof}
From the dimension count at the beginning of \cref{sec:Higher Pythagoras numbers} we get $\py_4(2,8)\ge 3$.

Let $f\in\interior\bin$ and let $K$ be a compact cone basis of $\bin$ containing $x^8$. For some $\lambda>0$ we have $\lambda f\in K$. Consider the line through $x^8$ and $\lambda f$. This line intersects the boundary of $K$ in $x^8$ and another point $g$. By \cref{cor:boundary_length} $g$ has length at most 3 and thus $\lambda f$ has length at most 4 as a convex combination of $g$ and $x^8$.
\end{proof}

Knowing the length of boundary points also allows us to prove that the boundary is irreducible.

\begin{prop}
\label{prop:linearly_dependent}
Let $D\subset\Rx_2^3$ be the subvariety containing linearly dependent triples. Then $\codim \phi_2^3(D)=2$.
\end{prop}
\begin{proof}
We parametrize the set of linearly dependent triples by $\R[x,y]_2^2\times\R^2$ as $(p_1,p_2,\lambda_1 p_1+\lambda_2 p_2)$.
\[
\phi'\colon \Rx_2^2\times\R^2\to \Rx_{8},\quad (p_1,p_2,\lambda_1,\lambda_2)\mapsto p_1^4+p_2^4+(\lambda_1p_1+\lambda_2p_2)^4
\]
The differential is given by ($h:=\lambda_1p_1+\lambda_2p_2$)
\begin{align*}
\dphi'(p_1,p_2,\lambda_1,\lambda_2)\colon &\Rx_2^2\times\R^2\to \Rx_{8},\\&(f,g,\mu_1,\mu_2)\mapsto 4f(p_1^3+h^3\lambda_1)+4g(p_2^3+h^3\lambda_2)+4h^3p_1\mu_1+4h^3p_2\mu_2.
\end{align*}
Hence, the image is $W:=\id{p_1^3+\lambda_1h^3,p_2^3+\lambda_2h^3,h^3p_1,h^3p_2}_8$ which we show has dimension at most 7. The dimension is maximal for a generic choice of $p_1,p_2,\lambda_1,\lambda_2$, hence we may assume that $p_1,p_2$ are linearly independent. Let $q\in \R[x,y]_2$ be such that $p_1,p_2,q$ form a basis of $\R[x,y]_2$ and let $V:=\spn(p_1^3+\lambda_1 h^3,p_2^3+\lambda_2 h^3)$. Then
\begin{align*}
W&=V\spn(p_1,p_2,q)+h_3\spn(p_1,p_2)\\
&=qV+\spn(p_1,p_2)\spn(V,h^3)\\
&=qV+\spn(p_1,p_2)\spn(p_1^3,p_2^3,h^3)\\
&\subset qV+\spn(p_1,p_2)^4
\end{align*}
As $\dim \spn(p_1,p_2)^4\le 5$ we see $\dim W\le 7$. Moreover, for a generic choice of $p_1,p_2,\lambda_1,\lambda_2$, the dimension of $W$ is equal to 7.
\end{proof}

Let $\mathcal{G}$ denote the Zariski-closure of the set
\[
\left\lbrace(q_1,q_2,q_3)\in\R[x,y]_2^3\colon q_1,q_2,q_3 \text{ linearly independent}, \id{q_1^3,q_2^3,q_3^3}_8\neq \R[x,y]_8\right\rbrace.
\]

\begin{thm}
\label{thm:algebraic_boundary}
The algebraic boundary of $\bin$ is an irreducible hypersurface. It is given as the Zariski-closure of the image of the hypersurface $\mathcal{G}$ under the sum of fourth powers map, i.e.
\[
\partial_a\bin=\ol{\phi_2^3(\mathcal{G})}.
\]
\end{thm}
\begin{proof}
Since every element of $\partial\bin$ has length 3 we can compute the boundary as (some of the) singular values of the map 
\[
\phi_2^3\colon \R[x,y]_2^3\to\R[x,y]_8,\quad (q_1,q_2,q_3)\mapsto q_1^4+q_2^4+q_3^4.
\]
The differential at a point $(q_1,q_2,q_3)$ is a linear map between 9-dimensional vector spaces of which we can compute the determinant with respect to the monomial bases. Considering the coefficients of $q_1,q_2,q_3$ as indeterminates, this determinant factors into two irreducible polynomials. One is the determinant of the $3\times 3$ coefficient matrix of the quadratic forms $q_1,q_2,q_3$. The other cuts out the hypersurface $\mathcal{G}$ which we may check is irreducible over $\C$ with {\tt Macaulay2} \cite{M2}.

By \cref{prop:linearly_dependent} the image of linearly dependent triples does not give rise to a component of the boundary.

Therefore, the algebraic boundary is given as the Zariski-closure of the image of the hypersurface $\mathcal{G}$ under the map $\phi_2^3$.
\end{proof}

\begin{rem}
Since we know that the algebraic boundary is irreducible, it is also given as the following dual variety by \cite[Theorem 3.7.]{sinn2015}. For some faces $\face$ of $\ter$ whose relative interior points consist of psd ternary quartics with exactly 3 real zeros the intersection $\face\cap U$ has dimension 1. A generator of this 1-dimensional cone spans an extreme ray of $\ter\cap U$. If $X$ denotes the variety given as the Zariski-closure of all such extreme rays, its dual variety $X^\star$ is precisely the algebraic boundary of $\bin$.
\end{rem}

\begin{rem}
\label{rem:different_number_of_length_3_reps}
%This proves that also the interior of $\bin\setminus \phi_2^3(D)$ is still connected.
%Thus if there exist forms of length 4, these are separated from forms of length 3 by $\ol{\phi_2^3(\mathcal{G})}$ only.
From numerical calculations in \texttt{Julia} \cite{beks2017} using the \texttt{HomotopyContinuation.jl} package \cite{hc} we know that there are binary forms with different numbers of real length 3 representations. Semi-algebraic subsets of $\bin$ where forms have different numbers of decompositions over $\R$ have to be separated by $\ol{\phi_2^3(\mathcal{G})}$ by \cref{prop:linearly_dependent} and \cref{thm:algebraic_boundary}. The computations also show that $\bin\setminus\ol{\phi_2^3(\mathcal{G})}$ does indeed have several connected components (assuming that all solutions have been found using Homotopy Continuation).

Up to sign and permutation of the summands $f_i$ has exactly $i$ real length 3 representations.
\begin{align*}
f_2&=(4x^2-5xy+2y^2)^4+(4x^2-4xy+2y^2)^4+(4x^2+2xy- y^2)^4\\
f_4&=(6x^2+10xy-2y^2)^4+(3x^2-2xy+5y^2)^4+(10x^2+4xy-6 y^2)^4\\
f_6&=(x^2+4xy+8y^2)^4+(10x^2+xy+3y^2)^4+(10x^2+10xy-10 y^2)^4
\end{align*}

The number of complex length 3 representations was constant at 76 for all (generic) examples computed.
The fact that all representations of the $f_i$'s that have been found numerically, give correct decompositions can be proven using the \texttt{cert} command in the \texttt{HomotopyContinuation} package.
This does not prove that all representations have been found even though this seems to be the case here as the same number of complex solutions $29184=384\cdot 76$ has been found in every example calculated. This number is the number of representations without accounting for permutations of the summands and scaling summands by roots of unity.
\end{rem}

\begin{example}
\label{ex:example_dual_element}
One example of a psd ternary quartic in the cone $(\bin)^\star=(\ter\cap U)$ is given by the following Gram matrix
\[\tiny
\left(\begin{array}{r}
-918256 c^{7} + 14398014 c^{6} - 284749896 c^{5} + 34054823830 c^{4} - 389117333868 c^{3} + 384885454002 c^{2} - 323875629964 c - 248202 \\
-659350 c^{7} + 20329939 c^{6} - 68106915 c^{5} - 4687738385 c^{4} + 48361507226 c^{3} - 47864048337 c^{2} + 39815341637 c + 30859 \\
-95223 c^{7} - 2098390 c^{6} + 35378289 c^{5} + 1998413271 c^{4} - 24797139347 c^{3} + 24583084470 c^{2} - 20789333367 c - 15937
\end{array}\right.
\]
\[\tiny
\begin{array}{r}
-659350 c^{7} + 20329939 c^{6} - 68106915 c^{5} - 4687738385 c^{4} + 48361507226 c^{3} - 47864048337 c^{2} + 39815341637 c + 30859\\
-823044 c^{7} + 6018130 c^{6} + 23370756 c^{5} + 30387414 c^{4} - 98730364 c^{3} + 120560334 c^{2} - 64520364 c - 50\\
4786 c^{7} - 790488 c^{6} + 6085491 c^{5} - 99186325 c^{4} + 1409481342 c^{3} - 1392614400 c^{2} + 1196060035 c + 927
\end{array}
\]
\[\tiny
\left.\begin{array}{r}
-95223 c^{7} - 2098390 c^{6} + 35378289 c^{5} + 1998413271 c^{4} - 24797139347 c^{3} + 24583084470 c^{2} - 20789333367 c - 15937\\
4786 c^{7} - 790488 c^{6} + 6085491 c^{5} - 99186325 c^{4} + 1409481342 c^{3} - 1392614400 c^{2} + 1196060035 c + 927\\
-5663 c^{7} - 26893 c^{6} - 858726 c^{5} + 128759213 c^{4} - 1359409919 c^{3} + 1343865411 c^{2} - 1122516338 c - 859
\end{array}\right)
\]
where $c$ is the real zero of 
\[
c^8 - c^7 + c^6 - 133c^5 - 135542c^4 + 1550171c^3 - 1534319c^2 + 1290239c + 1
\]
that is approximately $\minus 21$ and with respect to the basis $(q_1,q_2,q_3)$ with
\begin{align*}
q_1&=-(x+z)(cx-z)\\
q_2&=(x + z)((c-16)x-(c+1)y - 17z)\\
q_3&=-(cx - z)((c-16)x-(c+1)y - 17z).
\end{align*}
It vanishes at the real points $(0:1:0),(1:1:-1),(1:-16:c)$ and therefore exposes the 3-dimensional face
\[
\mathrm{cone}((xy)^4,(x^2+xy-y^2)^4,(x^2-16xy+cy^2)^4)
\]
of the cone $\bin$. The linear functional in $\R[x,y]_8^\du$ is a generator of the degree 8 component of the Gorenstein ideal $\id{(xy)^3,(x^2+xy-y^2)^3,(x^2-16xy+cy^2)^3}_8$, a complete intersection minimally generated in degree $5$.

We also give examples of psd ternary quartics in $U$ with exactly 1 and 2 real zeros. Wrt the basis $x(x+z),xy,z(x+z),yz$ the following Gram matrix defines a form in $(\bin)^\star$ with real zeros $(0:1:0)$ and $(1:0:-1)$:
\[
\left(\begin{array}{rrrr}
2 & 0 & \minus 1 & 0 \\
0 & 3 & 0 & 0 \\
\minus 1 & 0 & 2 & 0 \\
0 & 0 & 0 & 3
\end{array}\right).
\]
In particular, taking the sum of the two quartics with 3 and 2 real zeros gives a quartic in $(\bin)^\star$ with the real zero $(0:1:0)$.
\end{example}

%\begin{rem}
%\label{rem:gorenstein_components}
%Consider the irreducible hypersurface
%\[
%\mathcal{G}=\overline{\lbrace (q_1,q_2,q_3)\in\R[x,y]_2^3\colon q_1,q_2,q_3 \text{ linearly independent}, \id{\ul q^3}_8\neq \R[x,y]_8\rbrace}.
%\]
%The closure of the image $Z'$ of the rational map $\phi\colon \mathcal{G}\dashrightarrow\P\R[x,y]_8,\, (q_1,q_2,q_3)\mapsto \id{\ul q}_8^\perp$ is also irreducible.
%
%Consider the two subsets of $\R[x,y]_8$, $\Gor(5,5)$ and $\Gor(4,6)$ parametrizing Gorenstein ideals/complete intersection of socle degree 8 with minimal generators of degree $5,5$ and $4,6$ respectively. By \cite{ika1999} we have $\Gor(4,6)\subset\ol{\Gor(5,5)}$ and $\ol{\Gor(4,6)}\cap\Gor(5,5)=\emptyset$ and both are irreducible. Thus, either $Z'\subset\ol{\Gor(4,6)}$ or $Z'\cap\Gor(5,5)$ is dense in $Z'$. Since $Z'\cap\Gor(5,5)$ is non-empty by \cref{ex:example_dual_element} the second case holds.
%
%More concretely this means that if $\ul q\in \mathcal{G}$ is generic then there exist elements $p_1,p_2\in\R[x,y]_5$ such that $\id{q_1^3,q_2^3,q_3^3}_d=\id{p_1,p_2}_d$ for all degrees $d\ge 6$.
%\end{rem}

\begin{cor}
Let $f\in\partial\bin$ be such that $f\notin l^4\interior(\Sigma^4_{2,4})$ for any $l\in\R[x,y]_1$. Then $f$ has a unique sum of fourth powers decomposition (up to sign and permutation of the summands).
\end{cor}
\begin{proof}
If $f=q^4$ then $q$ is indefinite by \cref{lem:pd_impossible}. Hence, this is the only representation and we may now assume that $f$ has length at least 2.
If $f=l^4(q_1^4+q_2^4)\in l^4 \partial(\Sigma_{2,4}^4)$ the representation is unique by \cite[Thm 5.1.]{reznick1992}.

Thus we may now assume that there exists $F\in(\bin)^\star$ with finitely many zeros and such that $F(f)=0$. If $f=\sum_{i=1}^r p_i^4$ then $F$ vanishes at all points $p_i$. Let $q_1,\dots,q_k$ be binary quadratic forms corresponding to the zeros of $F$. Then the $p_i$ are scalar multiples of some of the $q_i$. Then any representation of $f$ can only use the $q_i$ as summands.
Assume $k=3$. If $f$ has at least 2 representations, then
\[
(\lambda_1q_1)^4+(\lambda_2q_2)^4+(\lambda_3q_3)^4=(\mu_1q_1)^4+(\mu_2q_2)^4+(\mu_3q_3)^4
\]
for some $\lambda_i,\mu_i\in\R$. Without loss of generality $\lambda_1^4\neq \mu_1^4$. Hence,
\[
q_1^4(\lambda_1^4-\mu_1^4)=(\mu_2q_2)^4+(\mu_3q_3)^4-(\lambda_2q_2)^4-(\lambda_3q_3)^4=c_1q_2^4+c_2q_3^4.
\]
If both $c_1,c_2$ have the same sign, both real linear factors of $q_1$ divide both $q_2$ and $q_3$, hence they are all the same and $f$ has length 1.
If they have different signs, then the right-hand-side factors over the reals as 
\[
c_1q_2^4+c_2q_3^4=\pm (\sqrt{|c_1|}q_2^2+\sqrt{|c_2|}q_3^2)(\sqrt[4]{|c_1|}q_2-\sqrt[4]{|c_2|}q_3)(\sqrt[4]{|c_1|}q_2+\sqrt[4]{|c_2|}q_3).
\]
However, now $q_1$ is a scalar multiple of the last two factors, showing that the three points $q_1,q_2,q_3$ are collinear which implies that $F$ has infinitely many zeros, a contradiction.
If one of the $c_i$ is equal to zero then two of the $q_i$ coincide, i.e. $F$ has only two zeros which is not true.

Lastly, consider the case $k=2$. Then any representation of $f$ can only use $q_1$ and $q_2$ and we get 
\[
(\lambda_1q_1)^4-(\mu_1q_1)^4=(\mu_2q_2)^4-(\lambda_2q_2)^4
\]
which shows that $q_1$ is a scalar multiple of $q_2$.
\end{proof}

\begin{rem}
The assumption $f\notin l^4\interior(\Sigma_{2,4}^4)$ is necessary as a general element of $\Sigma_{2,4}^4$ has a 1-dimensional family of length 3 representation which can be read off the map
\[
\phi_1^3\colon\R[x,y]_1^3\to\R[x,y]_4,\, (l_1,l_2,l_3)\mapsto l_1^4+l_2^4+l_3^4
\]
where the left-hand-side has dimension 6 and the right-hand-side has dimension 5.
\end{rem}

\section{Facial structure}
\label{sec:Facial structure}

We study the facial structure of the convex cone $\bin$.
We prove that there are exactly four types of exposed faces on $\bin$. Three of them are polyhedral, they are of the form $F_i:=\cone(q_1^4,\dots,q_i^4)$ with $i\in \{1,2,3\}$. And one is not polyhedral, it is of the form $l^4\Sigma_{2,4}^4$ for some $l\in\R[x,y]_1$. Moreover, there are some non-exposed faces contained in $l^4\Sigma_{2,4}^4$.

\begin{dfn}
Let $n\in\N$ and $C\subset\R^n$ be a convex cone. A \todfn{face} $F$ of $C$ is a convex subset of $C$ such that 
for any $x,y\in C$ whenever $x+y\in F$ both $x,y$ are contained in $F$.
A face $F$ is called \todfn{exposed} if there is an affine hyperplane $H$ such that $C\cap H=F$.
\end{dfn}

\begin{rem}
In general, not every face is exposed. However, this does hold if the cone $C$ is a polyhedron or more generally a spectrahedron \cite{rg1995}.
\end{rem}

\begin{prop}
Every exposed face of $\bin$ is of one of the four types $F_1,F_2,F_3$ or  $l^4\Sigma_{2,4}^4$ for some $l\in\R[x,y]_1$. Moreover, there are exposed faces of every such type.
\end{prop}
\begin{proof}
Let $F\in\partial(\bin)^\star$ be a psd ternary quartic. First assume that $\vc(F)(\R)$ is finite and $F$ has real zeros corresponding to the binary quadratic forms $q_1,\dots,q_i$ with $i\in\{1,2,3\}$. Then the face exposed by $F$ contains $\cone(q_1^4,\dots,q_i^4)$ and cannot contain any other form as we know all zeros of $F$. 
By \cref{ex:example_dual_element} there exist ternary quartics in $\partial(\bin)^\star$ with exactly one, two and three real zeros.

Now, assume $F$ has infinitely many real zeros. If $F(\sum_{i=1}^r q_i^4)=0$ and $F$ is singular at any of its real zeros we see that $F(q_i^3h)=0$ for all $i=1,\dots,r$ and all $h\in\R[x,y]_2$. By \cref{prop:infinite_zeros} $q_i$ is divisible by $l$ for some $l\in\R[x,y]_1$ and thus the face exposed by $F$ is contained in $l^4\Sigma_{2,4}^4$. Moreover, since $l^3\in L^\perp$ any form $l^3(lp)$ with $p\in \Sigma_{2,4}^4$ is contained in the face exposed by $F$.
\end{proof}

\begin{thm}
\label{thm:faces}
Let $0\neq \face\subsetneq\bin$ be a face. Then $\face$ is one of the following:
\begin{itemize}
\item an exposed face of the form $F_i=\cone(q_1^4,\dots,q_i^4)$ for some $i=1,2,3$ where all $q_j$ are indefinite binary quadratic forms,
\item an exposed face of type $l^4\Sigma_{2,4}^4$,
\item a non-exposed extreme ray of the form $\cone(l^8)=\R_+l^8$ with $l\in\R[x,y]_1$, or
\item a non-exposed face of the form $l^4\cone(l_1^4,l_2^4)$ with $l,l_1,l_2\in\R[x,y]_1$ and $l_1,l_2$ linearly independent.
\end{itemize}
\end{thm}
\begin{proof}
Let $f\in\partial\bin$ with a representation $f=\sum_{i=1}^j q_i^4$ and $j\le 3$. Let $\face\subset\bin$ be the supporting face of $f$, i.e. the smallest face containing $f$ in its relative interior.

$j=3$: If the three quadratic forms $q_1,q_2,q_3$ are linearly independent, they do not have a common factor and thus any ternary quartic $F\in(\bin)^\star$ that vanishes on $\face$ has exactly 3 real zeros by \cref{prop:infinite_zeros}. But then $F$ already exposes this face.
If the three quadratic forms are linearly dependent and are all divisible by $l\in\R[x,y]_1$, then $F$ already has to vanish on the line through all three points by Bezout's theorem, hence has infinitely many zeros. Since $f$ has a length 3 representation $f$ is contained in the interior of $l^4\Sigma_{2,4}^4$ and hence this face is equal to $\face$.

$j=1$: Any form $l^8$ and $(ll')^4$ defines an extreme ray since this representation as a sum of fourth powers is unique.
By \cref{lem:pd_impossible} $q^4$, $q\in\R[x,y]_2$ is not on the boundary of $\bin$ whenever $q$ is positive definite.

We claim that $l^8$ is not exposed and $(ll')^4$ is exposed. After a \cc\ $l^8=x^8$. We calculate
\[
U\cap I((1:0:0))_2^2=\spn(x y^{3} + 3 x^{2} y z, y^{2} z^{2} + \frac{2}{3} x z^{3}, y^{3} z + 3 x y z^{2}, y^{4} + 12 x y^{2} z + 6 x^{2} z^{2}, z^{4}, y z^{3}).
\]
The Gram matrices of an element 
\[
x y^{3} a_{1} + 3 x^{2} y z a_{1} + y^{2} z^{2} a_{2} + \frac{2}{3} x z^{3} a_{2} + y^{3} z a_{3} + 3 x y z^{2} a_{3} + y^{4} a_{4} + 12 x y^{2} z a_{4} + 6 x^{2} z^{2} a_{4} + z^{4} a_{5} + y z^{3} a_{6}
\]
in the intersection have the form
\[
\left(\begin{array}{rrrrr}
0 & 9 a_{1} & 3 a_{1} & 36 a_{4} - \lambda_{1} & 9 a_{3} - \lambda_{2} \\
9 a_{1} & 36 a_{4} & \lambda_{1} & \lambda_{2} & 2 a_{2} \\
3 a_{1} & \lambda_{1} & 6 a_{4} & 3 a_{3} & \lambda_{3} \\
36 a_{4} - \lambda_{1} & \lambda_{2} & 3 a_{3} & 6 a_{2} - 2 \lambda_{3} & 3 a_{6} \\
9 a_{3} - \lambda_{2} & 2 a_{2} & \lambda_{3} & 3 a_{6} & 6 a_{5}
\end{array}\right)
\]
for some $\lambda_1,\lambda_2,\lambda_3\in\R$ with respect to the basis $xy, xz, y^2, yz, z^2$.
By checking all principle $2\times 2$ minors, it is easy to see that the only form in the intersection having a psd Gram matrix is $a_5z^4$ with $a_5>0$.
This shows that $x^8$ and therefore $l^8$ is not exposed. Since exposed extreme rays are dense in the set of all extreme rays and $(ll')^4$ is one $\gl_2(\R)$ orbit they have to be exposed.

$j=2$: Assume the two summands have a common factor. Write $q_1=ll_1,\, q_2=ll_2$. We claim that the supporting face of $f=l^4(l_1^4+l_2^4)$ is given as $l^4\cone(l_1^4,l_2^4)$. This is true if every $g\in\cone(l_1^4,l_2^4)$ has a unique representation as a sum of fourth powers. Since $g\in\partial\Sigma_{2,4}^4$ all representations are of length at most 2. By \cite[Thm 5.1.]{reznick1992} such representations are unique up to permutation and sign.

Assume the summands do not have a common factor. Since $f$ is on the boundary there exists a psd ternary quartic $F\in(\bin)^\star$ vanishing at the two points. Since the two do not have a common factor, $F$ has at most 3 real zeros.
If $F$ has exactly these two real zeros, then $\face=\cone(q_1^4,q_2^4)$ and $\face$ is exposed by $F$.
If $F$ has 3 real zeros, the face exposed by $F$ is $\cone(q_1^4,q_2^4,q_3^4)$ for some $q_3\in\R[x,y]_2$. This is a polyhedral face with $f$ on its boundary. Thus $\face$ has dimension 2 and is given as $\cone(q_1^4,q_2^4)$.
\end{proof}

\begin{rem}
 Note that the structure of the extreme rays also implies that the algebraic boundary of the dual cone $(\bin)^\star$ is irreducible and is given as the dual variety of the Zariski-closure of the set $(ll')^4$ over all $l,l'\in\R[x,y]_1$ by \cite[Theorem 3.7.]{sinn2015}
\end{rem}

\begin{rem}
$\bin$ is not a spectrahedron due to the existence of non-exposed faces. It is of course a spectrahedral shadow since its dual cone is.
\end{rem}

\begin{rem}
\phantom{-}
\begin{center}
\begin{tikzpicture}
		\node  (0) at (0, 8) {$\bin$};
		\node  (1) at (-2, 6) {$l^4\Sigma_{2,4}^4$};
		\node  (2) at (2, 4) {$F_3$};
		\node  (3) at (-2, 2) {$l^4(l_1^4+l_2^4)$};
		\node  (4) at (2, 2) {$F_2$};
		\node  (5) at (-2, 0) {$l^8$};
		\node  (6) at (2, 0) {$(ll')^4$};
		\node  (7) at (6, 8) {9};
		\node  (8) at (6, 6) {5};
		\node  (9) at (6, 4) {3};
		\node  (10) at (6, 2) {2};
		\node  (11) at (6, 0) {1};
		\draw  (0) -- (1);
		\draw (1) to (3);
		\draw (3) to (5);
		\draw (0) to (2);
		\draw (2) to (4);
		\draw (4) to (6);
        \draw (3) to (6);
\end{tikzpicture}
\end{center}
The numbers on the right denote the dimension of faces in that line.

On the left-hand side, the two faces $l^4(l_1^4+l_2^4)$ and $l^8$ are not exposed, whereas $l^4\Sigma_{2,4}^4$ is exposed. On the right-hand side $F_3$ and $(l_1l_2)^4$ are exposed. There are faces of type $F_2$ that are exposed, however we do not know if this is always the case.
Except for $l^4\Sigma_{2,4}^4$ and $\bin$ itself all faces are polyhedral.

\end{rem}

We see from our description of the boundary that Reznick's Conjecture \cite[Conjecture 7.1]{reznick2011} does not hold. It states that every $f\in\bin$ is ``doubly positive" in the sense that it can be written as $f=f_1^2+f_2^2$ with $f_1,f_2\in\R[x,y]_4$ both themselves psd.

\begin{cor}
\label{cor:reznick_conjecture}
Let $f\in\partial\bin$ of length 3 and such that $f\notin l^4 \Sigma_{2,4}^4$ for any $l\in\R[x,y]_1$.
Then $f$ cannot be written as $f=f_1^2+f_2^2$ with $f_1,f_2\in\R[x,y]_4$ psd.

In particular, this holds for a generic $f\in\partial\bin$.
\end{cor}
\begin{proof}
Assume $f=f_1^2+f_2^2$ for some $f_1,f_2\in\R[x,y]_4$ psd.
Write $f_i=p_i^2+q_i^2$ with $p_i,q_i\in\R[x,y]_2$. If both $f_1,f_2$ are squares, we may assume $f=p_1^4+p_2^4$, and thus $f$ has length at most 2.

Assume $f_1$ is not a square and consider the identity
\begin{equation}
\label{eqn:length_3}
(p_1^2+q_1^2)^2=\frac{1}{18}((\sqrt{3}p_1+q_1)^4+(\sqrt{3}p_1-q_1)^4+16q_1^4).
\end{equation}
The same identity holds also with the roles of $p_1$ and $q_1$ swapped. Let $L\in(\bin)^\star$ with $L(f)=0$. Then by \cref{eqn:length_3} we see
\[
L(p_1^4)=L(q_1^4)=L((\sqrt{3}p_1+q_1)^4)=0.
\]
However, $p_1.q_1,\sqrt{3}p_1+q_1$ are three distinct collinear points in $\P^2(\R)=\P(\R[x,y]_2)$. Therefore, $F:=\bil{L}{q(a,b,c)}$ vanishes at infinitely many real points and thus $f\in l^4\Sigma_{2,4}^4$ for some $l\in\R[x,y]_1$ by \cref{thm:faces}, a contradiction.
\end{proof}

\begin{rem}
From \cref{cor:reznick_conjecture} we also see that there is indeed a full-dimensional \sa\ set of forms in the interior of $\bin$ that can not be written in this way. 
Indeed, the set of all forms that can be written in this way are given as the image of the map
\[
\R[x,y]_2^4\to \bin,\quad (p_1,p_2,q_1,q_2)\mapsto (p_1^2+q_1^2)^2+(p_2^2+q_2^2)^2
\]
which is a closed \sa\ set. Since it is not equal to $\bin$ their difference contains a full-dimensional \sa\ set.
\end{rem}

\section{Open problems}
\label{sec:Open problems}
In this last section we propose some open problems which might stimulate further research in this area.

\begin{question}
What is the value of $p_{2s}(\R[x])$? In particular, what is the value of $p_4(\R[x])$?
\end{question}
So far the only known lower bound is the $2s \leq p_{2s}(\R[x])$ as shown at the beginning of the Section 2.

Theorem of Becker \cite[Theorem 2.12]{becker1982} states that for any field $K$ we have $p_2(K) < \infty $ iff $p_{2s}(K)< \infty$ for some $s$ iff  $p_{2s}(K)< \infty$ for all $s$. The next conjecture asks for a generalization of the above theorem
\begin{conj}
Let $R$ be a commutative ring with identity. Then the following conditions are equivalent.
\begin{enumerate}
\item $p_2(R)<\infty$
\item $p_{2s}(R) <\infty$ for some $s$
\item $p_{2s}(R) <\infty$ for all $s$.
\end{enumerate}
\end{conj}
    If the above conjecture is true, then the problem of finiteness of higher Pythagoras numbers could be reduced to the finiteness of the 2-Pythagoras number. This is important and usually much easier, mostly because the behavior of quadratic forms is well understood compared to forms of higher degrees. It is worth mentioning that the above conjecture does hold for real valuation rings. This follows from  \cite[Theorem 5.17]{bp1996} with the simple observation that for valuation rings it is enough to consider units which are sums of $2s$-th powers. Also, this paper provides quite a variety of rings for which the above conjecture trivially holds since all of the quantities are infinite.

\begin{question}
We saw in \cref{rem:different_number_of_length_3_reps} that a general element of $\bin$ has (up to permutation and roots of unity) 76 complex representations of length 3. Where does this number come from? Can we construct all 76 of them, and possibly even show that some of them are real?

This might be similar to length three sum of squares representations of ternary quartics. These are known to be in one-to-one correspondence (up to orthogonal equivalence) to Steiner complexes associated to the complex curve.
\end{question}

\begin{question}
The proof of \cref{thm:at_most_3_zeros} is ultimately computational. It would be nice to understand the reason why such a ternary quartic cannot have 4 real zeros.
\end{question}

{\bf Acknowledgements}. 
We thank Greg Blekherman for suggesting the connection to ternary quartics which turned out to be a key point of view. Moreover, we thank Paul Breiding for help with early computations in \texttt{Julia}. Julian Vill was partially supported by DFG grant 314838170,
  GRK 2297, MathCoRe.
Part of this paper came to existence during the stay of Tomasz Kowalczyk at the Otto von Guericke University in Magdeburg. The author is very grateful for the hospitality. The authors also thank the referee for many numerous comments.
\bibliographystyle{plain}
\bibliography{refs.bib}

\newpage

\section*{{\tt Macaulay2} Code}

\begin{verbatim}
R=QQ[t,a0,a1,a2,b0,b1,b2,c0,c1,c2,e1,e2,e3,a,b,c];
U={a^2*b^2+a^3*c*2/3,a*b^3+a^2*b*c*3,b^2*c^2+2/3*a*c^3,b^3*c+a*b*c^2*3,
b^4+12*a*b^2*c+6*a^2*c^2,a^4,a^3*b,c^4,b*c^3}; --basis of the subspace U
X={a^4,a^3*b,a^3*c,a^2*b^2,a^2*b*c,a^2*c^2,a*b^3,a*b^2*c,a*b*c^2,a*c^3,
b^4,b^3*c,b^2*c^2,b*c^3,c^4}; --basis of R[a,b,c]_4
lab=(a1*b2-a2*b1)*a+(a2*b0-a0*b2)*b+(a0*b1-a1*b0)*c;
lac=sub(lab,{b0=>c0,b1=>c1,b2=>c2});
lbc=sub(lab,{a0=>c0,a1=>c1,a2=>c2});
q1=lab*lac; q2=lab*lbc; q3=lbc*lac; --generators of I_2                                                                                      
L={q1^2,q1*q2,q1*q3,q2^2,q2*q3,q3^2}; --basis of I_2^2       
V=matrix{join(U,L)};
(M,C)=coefficients(V,Variables=>{a,b,c},Monomials=>X);

Adjugate = (G) -> (
n := degree target G;
m := degree source G;
matrix table(n, n, (i, j) -> (-1)^(i+j) * det(submatrix(G,
                                {0..j-1, j+1..n-1},{0..i-1, i+1..m-1})))
);

getEquations = (p1,p2,p3,p4,p5,p6,p7,p8,p9) -> (
G1:=sub(C,{a0=>p1,a1=>p2,a2=>p3,b0=>p4,b1=>p5,b2=>p6,c0=>p7,c1=>p8,c2=>p9});
G1ad := Adjugate G1;
de := det G1;
G := 1/2*matrix {{2*G1ad_(9,2), G1ad_(10,2), G1ad_(11,2) },{G1ad_(10,2),
     2*G1ad_(12,2), G1ad_(13,2)},{G1ad_(11,2), G1ad_(13,2), 2*G1ad_(14,2)}}; 
--this is the 3x3 Gram matrix
cp := det (t*id_(R^3)-G);
cc := coefficients(cp,Variables=>t,Monomials=>{t^2,t});
{de,det G,(cc#1)_(0,0),(cc#1)_(1,0)} --output the two determinants that should 
--vanish, the coefficient of the characteristic polynomial of G that needs to be 
--negative in order for G to be psd and the coefficient that we saturate with
);

allEq = (p1,p2,p3,p4,p5,p6,p7,p8,p9) -> (
E := getEquations(p1,p2,p3,p4,p5,p6,p7,p8,p9);
I := ideal(radical(ideal(E#0)),radical(ideal(E#1))); --we take radicals to 
--remove unnecessary powers of some factors, this is not relevant for
--the minimal primes
J := saturate(I,radical(ideal(E#3)));
Jmin := minprimes J;
{toExternalString Jmin,toExternalString E#2,toExternalString E#3}
);
E1=allEq(0,1,0,1,e1,1,1,e2,e3); --q_i lin.ind. gives condition e_3 != 1
E2=allEq(0,1,0,1,e1,-1,1,e2,e3); --q_i lin.ind. gives condition e_3 != -1
\end{verbatim}

\end{document}